\documentclass[12pt]{article}
\usepackage{amsmath,amssymb,bm}
\usepackage{color,url}
\usepackage{xcolor} % e.g. for grey
\usepackage[numbers]{natbib} % for references
\usepackage{diagbox} % for tables
\usepackage{multirow} % for tables
\usepackage{dsfont} % for \mathbb{1}
\usepackage{graphicx} % for images
\usepackage{caption} % for long captions
\newtheorem{theorem}{Theorem}% [section]
\newtheorem{definition}{Definition}
\newtheorem{corollary}{Corollary}
\newtheorem{lemma}{Lemma}

\newtheorem{remark}{Remark}

\def\bee{\begin{equation}}
\def\ene{\end{equation}}
\def\beea{\begin{eqnarray}}
\def\enea{\end{eqnarray}}
\def\beeas{\begin{eqnarray*}}
\def\eneas{\end{eqnarray*}}

\newenvironment{proof}{\begin{list}{$\!\!${\bf Proof.}%
  \rule{1pt}{0pt}}{\setlength{\leftmargin}{0pt}%
  \setlength{\itemindent}{30pt}%
  \setlength{\listparindent}{15pt}}\item}{\rule{0.3em}{0mm}%
  \hfill\framebox[1.2ex]{\rule{0.3em}{0mm}}\end{list}}

\def\ignore#1{}

\def\diag{\hbox{diag}}
\def\Diag{\hbox{Diag}}
\def\noise{{\bm{\xi}}}
\def\T{\text{T}}
\def\E{\mathds{E}}
\def\Cov{\text{Cov}}
\def\l{{\,\ell}}
\def\rt{\tilde{r}}
\def\tr{\text{tr}\,}
\def\rk{\text{rk}\,}
\def\R{\mathds{R}}
\def\kmax{{k_{\text{max}}}}
\def\xbar{\bar{x}}

\title { {\bf Expected Value of Matrix Quadratic Forms\\ with Wishart distributed Random Matrices} }

\author{\small Melinda Hagedorn, Heinrich Heine University, D\"usseldorf, Germany}
\date  { Dec. 13, 2022\medskip\\
All data generated or analysed in this article are available in \cite{github}.
}

\begin{document}
\maketitle 

\begin{abstract}
To explore the limits of a stochastic gradient method, it may be useful to consider an
example consisting of  an infinite number of quadratic functions. In this context, it is appropriate to determine the expected value and the covariance matrix of the stochastic noise, i.e. the difference of the true gradient and the approximated gradient generated from a finite sample. When specifying the covariance matrix, the expected value of a quadratic form $QBQ$ is needed, where $Q$ is a Wishart distributed random matrix and $B$ is an arbitrary fixed symmetric matrix. After deriving an expression for $\E(QBQ)$ and considering some special cases, a numerical example is used to show how these results can support the comparison of two stochastic methods.
\end{abstract}


\section{Outline}
The Wishart distribution is a generalization of the $\chi^2$ distribution. According to \cite{hardle2015multivariate} and \cite{kanti1979multivariate} the Wishart distribution plays a prominent role in estimating the covariance matrix in context of multivariate statistics. Therefore, it is not surprising that this important distribution is subject of current research. For instance, \cite{mathew1997wishart} considers quadratic forms $Y^\T CY$ with non-negative definite matrix $C$ and normally distributed random matrix $Y$ and investigates what are the necessary and sufficient conditions for $Y^\T CY$ to be Wishart distributed. Based on this, \cite{masaro2003wishart} examines in the special case of $Y$ with expected value zero under which conditions $Y^\T CY$ is central Wishart distributed. Furthermore, in \cite{neudecker1985dispersion} the dispersion matrix of $\text{vec}(Y^\T AY)$ is derived, where $A$ is an arbitrary nonrandom matrix. \\
In this paper we are interested in a different kind of quadratic form: For a Wishart distributed $Q$ and a symmetric matrix $B$ we derive an expression for the expected value of $QBQ$. For $B=I_n$ this is the second momentum of the Wishart distribution and thus part of the examination of the momenta of the Wishart distribution in \cite{bishop2018introduction}.\bigskip\\
In \cite{kollo2005advanced} a different and more general formula for the expected value of $XAX^\T\otimes XBX^\T$
was already derived, where $X\sim N_{n,k}(\mu, \Sigma, \Psi)$ with symmetric positive definite matrices $\Sigma$ and $\Psi$. While the formulation in \cite{kollo2005advanced} is mathematically equivalent to the one derived in this paper for the special case considered here, the actual computation of the expected value is quite different - using a linear system based on Kronecker products in \cite{kollo2005advanced} and a lower dimensional variant below. 
\bigskip\\
This paper is structured as follows. First, in chapter \ref{motivation}, 
we motivate in the context of the stochastic gradient method why
an expression for the expected value $\E(QBQ)$ is needed. In the \ref{introduction} chapter we recall  two important, well-known properties of Wishart distributed random matrices. With this preparation, we are then able to present a theorem with a general expression for $\E(QBQ)$ in the \ref{main} chapter, prove the assertion, and derive more compact expressions under stronger assumptions. Also the connection to the result of \cite{kollo2005advanced} is worked out in more detail. Finally, we show in chapter \ref{example} that the approximated value for increasing sample size approaches the theoretical value from the previous chapter, which illustrates the statement of the theorem, and use the theorem to compare the ordinary stochastic gradient method with a variant that uses averaging.
\bigskip\\
An application of the expected value $\E(QBQ)$ derived here is the minimization
of a random convex quadratic function. While convex quadratic problems in some
form are the simplest nontrivial problems, they are complex enough to reproduce
local dynamics of more difficult smooth problems. They arise in practical applications
in the form of large scale systems of linear equations and least squares problems. Studying the performance of a method on convex quadratic problems is
a fundamental preparation to extend the method to more general problems (see \cite{gonzaga2016steepest} and \cite{goh2017why}).

\section{Notation}
In this paper the all-one and the all-zero vectors and matrices are denoted by
\begin{align*}
 0_n &:= (0,\dots,0)^\T\in\mathds{R}^n,\  0_{n\times n} := 0_n0_n^\T\in\mathds{R}^{n\times n},\\
 \mathds{1}_n &:=(1,\ldots,1)^\T\in\mathds{R}^{n},\ \mathds{1}_{n\times n} := \mathds{1}_n \mathds{1}_n^\T\in\mathds{R}^{n\times n}
\end{align*}
and the identity matrix by $I_n$ with dimension $n\in\mathds{N}$.
The Hadamard product of two matrices $X$ and $Y$ of the same dimension is defined componentwise as $(X\circ Y)_{ij}:=X_{ij}Y_{ij}$. Let $X\otimes Y$ be the Kronecker product of two arbitrary matrices $X$ and $Y$. A matrix $M$ has rank $\rk(M)$, determinant $\text{det}(M)$ and trace $\tr(M):=\sum_{i=1}^n M_{i,i}$. If a matrix $M$ is positive definite, we write $M\succ0$. The vector with the diagonal elements of a quadratic matrix $M$ is denoted by $\diag(M)$ and for a vector $x\in\mathds{R}^n$ the expression $\Diag(x)$ symbolizes the $n\times n$ diagonal matrix with the entries of $x$ on its diagonal. The vector $\text{vec}(M)$ is obtained by stacking the columns of $M$ on top of one another. The inverse function of vec is $\text{mat}:=\text{vec}^{-1}$. Furthermore, $S_+^n$ denotes the set of all symmetric, positive definite matrices, i.\,e. 
\begin{align}
S_+^n:=\{M\in\mathds{R}^{n\times n}\ |\ M=M^\T,\ M\succ 0\}.
\end{align}
The expected value and the covariance matrix of a random vector $X$ are denoted by $\E(X)$ and $\Cov(X)$ whenever they exist.

\section{Motivation}\label{motivation}
\subsection{Stochastic Gradient Method}
In order to find the minimum of a function $f:\ \mathds{R}^n\to\mathds{R},\ f(x):=\tfrac{1}{m}\sum_{i=1}^m f_i(x)$, i.\,e. the root of $\nabla f$, the gradient descent can be used whose iterates $x^{k+1}=x^k-\gamma_k\nabla f(x^k)$ are generated with step length $\gamma_k\geq0$ starting at a point $x^0$. If $m$ is very large, the calculation of the exact gradient $\nabla f(x^k)$ is computationally expensive. To avoid computing the full gradient $\nabla f(x^k)$ at each iteration, it can be approximated. Assuming an i.\,i.\,d. chosen batch $S_k$ from the uniform distribution of $\{1,\ldots,m\}$, the expected value of $\nabla_{S_k}  f(x^k):= \frac 1{|S_k|} \sum_{i\in S_k}\nabla f_i(x^k)$ is 
\begin{align*}
 E(\nabla_{S_k}  f(x^k)) &= E\left(\frac 1{|S_k|} \sum_{i\in S_k}\nabla f_i(x^k)\right) 
= \frac 1{|S_k|} \sum_{i\in S_k}E(\nabla f_i(x^k))\\ 
&= \frac 1{|S_k|} \sum_{i\in S_k}\sum_{j=1}^m\nabla f_j(x^k)\cdot\text{P}(\nabla f_j(x^k)=\nabla f_i(x^k))
= \frac 1{|S_k|} \sum_{i\in S_k}\frac{1}{m}\sum_{j=1}^m\nabla f_j(x^k)\\
&= \frac 1{|S_k|}\cdot |S_k|\cdot \nabla f(x^k) 
= \nabla f(x^k).
\end{align*}
This is the motivation for using the approximation $\nabla_{S_k} f(x^k)$ instead of $\nabla f(x^k)$, i.\,e. the stochastic gradient (descent) method (SGD) is given by the iterates
\begin{align} \label{proc1}
x^{k+1} = x^k-\gamma_k\nabla_{S_k}f(x^k).
\end{align}
There exist numerous modifications of the stochastic gradient method. The following example can be useful to examine the limits of a SGD method or to compare two variants of SGD.
\subsection{Random quadratic functions}
As in \cite{averaging}, we assume a fixed matrix $A\in\mathds{R}^{n\times n}$ with $\det(A)\neq 0$ and $n\in\mathds{N}$. At each iteration $\ell\in\{1,\dots,m\}$ we draw random vectors $r^\l$ and $b^\l$ independently from the $n$-variate normal distribution with expected value $0_n$ and covariance matrix $\Sigma=\Sigma^\T\succ 0$. Briefly, this can be written as $r^\l,b^\l\sim N_n(0_n,\Sigma)$. With $a^\l:=Ar^\l$ we are able to define the functions
\begin{align}\label{fl}
 f_\ell:\ \mathds{R}^n\to\mathds{R},\ f_\ell(x):=\tfrac12((a^\l)^\T x)^2 + (b^\l)^\T x.
\end{align}
Since $a^\l$ has the expected value $\E(a^\l)=\E(Ar^\l)=A\E(r^\l)=0_n$ and the covariance matrix $\Cov(a^\l)=\Cov(Ar^\l)=A\Cov(r^\l)A^\T=A\Sigma A^\T$, it holds $a^\l\sim N_n(0_n,A\Sigma A^\T)$
 and the second momentum is given by $\E(a^\l(a^\l)^\T)=\Cov(a^\l)+\E(a^\l)\E(a^\l)^\T=\Cov(a^\l)=A\Sigma A^\T$. 
Because $f_\ell$ are quadratic functions of $r^\l$ and $b^\l$ with expected value
\begin{align*}
 \E(f_\ell(x)) =\tfrac12\E(x^\T a^\l (a^\l)^\T x)+x^\T\E(b^\l) 
 = \tfrac12 x^\T\E(a^\l(a^\l)^\T)x
 = \tfrac12x^TA\Sigma A^\T x =: f(x)
\end{align*}
and due to the existence of the fourth momenta of $r^\l$ and $b^\l$, for a given $x$ the variances of $f_\ell(x)$ are bounded and almost surely it exists
\begin{align}\label{objective}
 \lim_{m\to\infty}\sum_{\ell=1}^m f_\ell(x)=f(x).
\end{align}
The stochastic gradient method uses the approximation $\nabla f_\ell(x)$ instead of the full gradient $\nabla f(x)$. Therefore it is reasonable to examine the noise $\noise^\l$ defined by
\begin{align}\label{noise}
 \noise^\l := \nabla f_\ell(x)-\nabla f(x) = a^\l(a^\l)^\T x + b^\l - A\Sigma A^\T x
\end{align}
with expected value
\begin{align*}
 \E(\noise^\l) = \E(a^\l(a^\l)^\T)x+\E(b^\l)-A\Sigma A^\T x
 = A\Sigma A^\T x +0_n-A\Sigma A^\T x = 0_n.
\end{align*}
Using the independence of $a^\l$ and $b^\l$ and defining $B:=A^\T xx^\T A$, the covariance matrix of $\noise^\l$ can be written as 
\begin{align}
 \Cov(\noise^\l)
 &= \E(\noise^\l(\noise^\l)^\T)-\E(\noise^\l)\E(\noise^\l)^\T
 = \E(\noise^\l(\noise^\l)^\T)\nonumber\\
 &= \E((a^\l(a^\l)^\T x + b^\l - A\Sigma A^\T x)(a^\l(a^\l)^\T x + b^\l - A\Sigma A^\T x)^\T)\nonumber\\
 &= \E(a^\l(a^\l)^\T xx^\T a^\l(a^\l)^\T) + \E(a^\l(a^\l)^\T x(b^\l)^\T) - \E(a^\l(a^\l)^\T xx^\T A\Sigma A^\T)\nonumber\\
 &\ \ \ \, + \E(b^\l x^\T a^\l(a^\l)^\T) +\E(b^\l(b^\l)^\T) - \E(b^\l x^\T A\Sigma A^\T)\nonumber\\
 &\ \ \ \, - \E(A\Sigma A^\T xx^\T a^\l(a^\l)^\T) - \E(A\Sigma A^\T x(b^\l)^\T) + \E(A\Sigma A^\T xx^\T A\Sigma A^\T)\nonumber\\
  &= \E(Ar^\l(r^\l)^\T A^\T xx^\T A r^\l(r^\l)^\T A^\T) + \E(a^\l(a^\l)^\T x)\E(b^\l)^\T - \E(a^\l(a^\l)^\T) xx^\T A\Sigma A^\T\nonumber\\
 &\ \ \ \, + \E(b^\l)\E( x^\T a^\l(a^\l)^\T) +\left[ \Cov(b^\l)+\E(b^\l)\E(b^\l)^\T \right] - \E(b^\l) x^\T A\Sigma A^\T\nonumber\\
 &\ \ \ \, - A\Sigma A^\T xx^\T \E(a^\l(a^\l)^\T) - A\Sigma A^\T x\E(b^\l)^\T + A\Sigma A^\T xx^\T A\Sigma A^\T\nonumber\\
 &= A\E(r^\l(r^\l)^\T B r^\l(r^\l)^\T) A^\T + 0_{n\times n} - A\Sigma A^\T xx^\T A\Sigma A^\T + 0_{n\times n} + \Sigma - 0_{n\times n}\nonumber\\
 &\ \ \ \,- A\Sigma A^\T xx^\T A\Sigma A^\T - 0_{n\times n} + A\Sigma A^\T xx^\T A\Sigma A^\T \nonumber\\
 &= A\E(r^\l(r^\l)^\T B r^\l(r^\l)^\T) A^\T + \Sigma - A\Sigma B\Sigma A^\T.\label{covxi}
\end{align}
This motivates us to determine expected values of the form $\E(r^\l(r^\l)^\T B r^\l(r^\l)^\T)$ with random vectors $r^\l\sim N_n(0_n,\Sigma)$ and a symmetric matrix $B$.

\section{Introduction}\label{introduction}
In the context of multivariate statistics, the following definition is of great importance:
\begin{definition}\label{defwishart}
 For $k\in\mathds{N}$ independent and identically distributed (i.\,i.\,d.) random vectors $r^\l\sim N_n(0_n,\Sigma)$ with covariance matrix $\Sigma=\Sigma^\T\succ0$ the random matrix $Q:=\sum_{\ell=1}^k r^\l(r^\l)^\T$
is called $n$-variate Wishart distributed with scale matrix $\Sigma$ and $k$ degrees of freedom. We write $Q\sim W_n(\Sigma, k)$. 
\end{definition}
\begin{remark}
 Alternatively $Q\sim W_n(\Sigma,k)$ can be defined by $Q:=RR^\T$, where $R$ is a random $n\times k$ matrix, which columns $r^\l$ are independent and identically $N_n(0_n,\Sigma)$ distributed. This coincides exactly with Definition \ref{defwishart}.
%  \begin{align*}
%   Q&=RR^\T
%   = \begin{pmatrix} (r^1)_1 & (r^2)_1 & & (r^k)_1\\ \vdots&\vdots&\cdots&\vdots \\ (r^1)_n & (r^2)_n & & (r^k)_n  \end{pmatrix} \begin{pmatrix} (r^1)_1& \cdots & (r^1)_n\\ (r^2)_1 &\cdots & (r^2)_n\\ & \vdots & \\ (r^k)_1 & \cdots & (r^k)_n \end{pmatrix}\\
%   &= \begin{pmatrix} \sum_{\ell=1}^k(r^\l)_1^2 & \sum_{\ell=1}^k (r^\l)_1(r^\l)_2 & \cdots &\sum_{\ell=1}^k (r^\l)_1(r^\l)_n\\ \sum_{\ell=1}^k (r^\l)_1(r^\l)_2 & \sum_{\ell=1}^k (r^\l)_2^2&&\vdots\\ \vdots&&\ddots&\\ \sum_{\ell=1}^k (r^\l)_1(r^\l)_n&\cdots&&\sum_{\ell=1}^k (r^\l)_n^2 \end{pmatrix}\\
%   &= \sum_{\ell=1}^k \begin{pmatrix} (r^\l)_1^2 & (r^\l)_1(r^\l)_2 & \cdots & (r^\l)_1(r^\l)_n\\ (r^\l)_1(r^\l)_2 & (r^\l)_2^2&&\vdots\\ \vdots&&\ddots&\\ (r^\l)_1(r^\l)_n&\cdots&& (r^\l)_n^2 \end{pmatrix}
% = \sum_{\ell=1}^k r^\l(r^\l)^\T.
%  \end{align*}
\end{remark}
The following result is well known:
\begin{lemma}\label{expwishart}
 The expected value of $Q\sim W_n(\Sigma,k)$ is 
 \begin{align*}
 \E(Q) = k\Sigma.
 \end{align*}
\end{lemma}
In the proof of the main result, the following important property of Wishart distributed random matrices will also be needed. For the sake of completeness it is proved here.
\begin{lemma}\label{CTQC}
 Consider $Q\sim W_n(\Sigma,k)$ with $\Sigma\in S_+^n$, $C\in\mathds{R}^{n\times m}$ and $\rk(C)=m\in\mathds{N}$. Then
 \begin{align*}
  C^\T QC \sim W_m(C^\T \Sigma C, k).
 \end{align*}
\end{lemma}
\begin{proof}
 Since we assume $Q\sim W_n(\Sigma,k)$, there exist $k$ independent and identically distributed $r^\l\sim N_n(0_n,\Sigma)$ so that $Q=\sum_{\ell=1}^k r^\l(r^\l)^\T$. Consider $C^\T r^\l$ with the expected value $\E(C^\T r^\l)=C^\T \E(r^\l)=0_m$ and positive definite $\Cov(C^\T r^\l)=C^\T \Cov(r^\l) C = C^\T\Sigma C$ due to the full rank of $C$, thus $C^\T r^\l\sim N_m(0_m,C^\T\Sigma C)$. Then we are able to conclude that
 $C^\T QC = C^\T\left( \sum_{\ell=1}^k r^\l(r^\l)^\T \right)C = \sum_{\ell=1}^k (C^\T r^\l)(C^\T r^\l)^\T\sim W_m(C^\T\Sigma C,k)$.
\end{proof}

\section{Determination of $\E(QBQ)$}\label{main}
The following theorem is the main result of this paper:
\begin{theorem}\label{expqbq}
 For $n\in\mathds{N}$ we consider the $n$-variate Wishart distributed random matrix $Q$ with the symmetric, positive definite scale matrix $\Sigma$ and $k\in\mathds{N}$ degrees of freedom, i.\,e. $Q\sim W_n(\Sigma,k)$. Furthermore, we deal with %the diagonal matrix $D$ and the orthogonal matrix $U$ from an eigendecomposition $\Sigma=UDU^\T$ and with 
 a fixed matrix $B=B^\T\in\mathds{R}^{n\times n}$. The expected value of the quadratic form $QBQ$ is
 \begin{align*}
\E(QBQ) %= kU\left[ 2\left(\diag(D)(\diag(D))^\T\,\right)\circ (U^\T BU)+\tr(U^ \T BUD)D \right] U^\T+(k^2-k)\Sigma B\Sigma.
= k\cdot\tr(B\Sigma)\Sigma + (k^2+k)\Sigma B\Sigma.
 \end{align*}
\end{theorem}
\begin{proof}
 Let the matrix $\Sigma=\Sigma^\T\succ0$ be factorized as $\Sigma=UDU^\T$ with diagonal matrix $D$ and orthogonal matrix $U$, i.\,e. $UU^\T=I_n$. Since $Q\sim W_n(\Sigma,k)$, there exist i.\,i.\,d. random vectors $r^\l\sim N_n(0_n,\Sigma)$ so that $Q=\sum_{\ell=1}^k r^\l(r^\l)^\T=\sum_{\ell=1}^k Q_\ell$ with $Q_\ell := r^\l(r^\l)^\T\sim W_n(\Sigma,1)$. We note that Lemma \ref{expwishart}
 gives us $\E(Q_\ell)=1\cdot\Sigma = \Sigma$.
 Now we define the unitary transformations 
 \begin{align*}
 \tilde{B}:=U^\T BU,\ \ \rt^\l:=U^\T r^\l\ \text{ and }\ \tilde{Q}_\ell:=U^\T Q_\ell U.
 \end{align*}
 Since $\rt^\l\sim N_n(0_n,D)$, its components $\rt_i^\l\sim N(0,D_{i,i})$ have the momenta $\E(\rt_i^\l)=0$, $\E((\rt_i^\l)^2)=D_{i,i}$ and $\E((\rt_i^\l)^4)=3 D_{i,i}^2$, and due to Lemma \ref{CTQC} we find
 \begin{align*}
  \tilde{Q}_\ell = U^\T Q_\ell U =U^\T r^\l(r^\l)^\T U = U^\T r^\l(U^\T r^\l)^\T = \rt^\l(\rt^\l)^\T\sim W_n(D,1).
 \end{align*}
 The expected value of $\tilde{Q}_\ell\tilde{B}\tilde{Q}_\ell$ is componentwise given by
 \begin{align*}
  \E(\tilde{Q}_\ell\tilde{B}\tilde{Q}_\ell)_{i,j} 
  &= \E(e_i^\T \rt^\l(\rt^\l)^\T\tilde{B}\rt^\l(\rt^\l)^\T e_j)
  = \E(\rt_i^\l\rt_j^\l (\rt^\l)^\T\tilde{B}\rt^\l)
  = \E\left( \rt_i^\l\rt_j^\l \sum_{p,q=1}^n \tilde{B}_{p,q} \rt_p^\l \rt_q^\l \right)\\
  &= \sum_{p,q=1}^n \tilde{B}_{p,q} \E(\rt_i^\l \rt_j^\l \rt_p^\l \rt_q^\l)\\
  &= \sum_{p,q=1}^n \tilde{B}_{p,q}\left\{  
  \begin{tabular}{ll}
  $3D_{i,i}^2$ &if  $i=j=p=q$\\
  $D_{i,i}D_{p,p}$ &if $i=j\neq p=q$\\ 
  $D_{i,i}D_{j,j}$ &if $i\neq j\ \text{and } ((i=p,j=q)\ \text{or }(i=q,j=p))$\\
  0 &else
  \end{tabular}
  \right.\\
  &= \left\{  
  \begin{tabular}{ll}
  $3D_{i,i}^2\tilde{B}_{i,i} + D_{i,i}\sum_{p=1, p\neq i}^n\tilde{B}_{p,p}D_{p,p}$ &if $i=j$\\
  $2D_{i,i}D_{j,j}\tilde{B}_{i,j}$ &if $i\neq j$
  \end{tabular}
  \right.\\
  &= \left\{  
  \begin{tabular}{ll}
  $2D_{i,i}^2\tilde{B}_{i,i} + D_{i,i}\tr(\tilde{B}D)$ &if $i=j$\\
  $2D_{i,i}D_{j,j}\tilde{B}_{i,j}$ &if $i\neq j$
  \end{tabular}
  \right.\\
  &= \left[ 2(\diag(D)\diag(D)^\T)\circ\tilde{B} +\tr(\tilde{B}D)D \right]_{i,j}.
 \end{align*}
All in all, we obtain
\begin{align*}
 \E(QBQ) &= \E\left( \left(\sum_{\ell=1}^k Q_\ell\right) B\left(\sum_{h=1}^k Q_h \right)\right) 
 = \sum_{\ell,h=1}^k\E(Q_\ell BQ_h)\\
 &= \sum_{\ell=1}^k\E(Q_\ell BQ_\ell) + \sum_{\ell,h=1,\ell\neq h}^k\E(Q_\ell BQ_h)\\
 &= \sum_{\ell=1}^k\E(UU^\T Q_\ell UU^\T BUU^\T Q_\ell UU^\T ) + \sum_{\ell,h=1,\ell\neq h}^k\E(Q_\ell) B\E(Q_h)\\
 &= \sum_{\ell=1}^k U\E(\tilde{Q}_\ell \tilde{B} \tilde{Q}_\ell)U^\T + \sum_{\ell,h=1,\ell\neq h}^k \Sigma B\Sigma\\
 &= kU\left[ 2(\diag(D)\diag(D)^\T)\circ\tilde{B} +\tr(\tilde{B}D)D \right] U^T + (k^2-k)\Sigma B\Sigma.
\end{align*}
With $\tilde{B}=U^\T BU$ as defined above one gets
\begin{align*}
 \E(QBQ) &= 2kU(D\tilde{B}D)U^\T +k\cdot\tr(\tilde{B}D)UDU^\T + (k^2-k)\Sigma B\Sigma\\
 &= 2k (UDU^\T) B(U DU^\T) + k\cdot\tr(U^\T BUD)\Sigma + (k^2-k)\Sigma B\Sigma\\
 &= 2k\Sigma B\Sigma + k\cdot\tr(BUDU^\T)\Sigma+(k^2-k)\Sigma B\Sigma
 = k\cdot\tr(B\Sigma)\Sigma+(k^2+k)\Sigma B\Sigma,
\end{align*}
which corresponds exactly to the assertion.
\end{proof}
\bigskip
In practice, the following special cases might be of interest. Additionally to the assumptions of Theorem \ref{expqbq} we demand $k=1$. Then the result of the theorem simplifies to
\begin{align}\label{k1}
 \E(QBQ) =  \tr(B\Sigma)\Sigma+2\Sigma B\Sigma
 %U\left[ 2\left(\diag(D)(\diag(D))^\T\,\right)\circ (U^\T BU)+\tr(U^\T BU D)D \right] U^\T.
\end{align}
If we assume in addition that 
%$\Sigma$ is a diagonal matrix, the eigendecomposition becomes trivial with $D\equiv\Sigma$ and $U\equiv I_n$. Thus,
%\begin{align}
% \E(QBQ) = 2 \left(\diag(D)(\diag(D))^\T\,\right)\circ B+\tr(BD)D.
%\end{align}
%The expression can be simplified even further by assuming 
$\Sigma=\sigma^2I_n$, we get
\begin{align}
 \E(QBQ) 
 %= 2 \sigma^4\mathds{1}_{n\times n}\circ B + \sigma^4\tr(BI_n)I_n 
 = \sigma^4  \left[2B+\tr(B)I_n\right].
\end{align}
Finally, the assumption $\Sigma=I_n$ leads to
%\begin{align}
% \E(QBQ) = 2 B+\tr(B)I_n ,
%\end{align}
%which corresponds to 
the special case analyzed in paper \cite{averaging}. With $B=I_n$ it follows immediately from Theorem \ref{expqbq}:
\begin{corollary}
 The second momentum of a $W_n(\Sigma,k)$ distributed random matrix $Q$ with scale matrix $\Sigma\in S_+^n$ is given by
  \begin{align*}
  \E(Q^2) = (k^2+k) \Sigma^2 + \tr(\Sigma)k\Sigma.
 \end{align*}
\end{corollary}
% \begin{proof}
% The corollary follows immediately from Theorem \ref{expqbq} with $B=I_n$:
% \begin{align*}
%  \E(Q^2) &= \E(QI_nQ) = kU\left[ 2\left(\diag(D)(\diag(D))^\T\,\right)\circ (U^\T U)+\tr(U^ \T UD)D \right] U^\T+(k^2-k)\Sigma^2\\
%  &= kU\left[ 2\left(\diag(D)(\diag(D))^\T\,\right)\circ I_n+\tr(D)D \right] U^\T+(k^2-k)\Sigma^2\\ 
%  &= kU\left[ 2 D^2+\tr(D)D \right] U^\T+(k^2-k)\Sigma^2\\
%  &= 2k UDU^\T U DU^T + \tr(D)kUDU^\T +(k^2-k)\Sigma^2\\
%  &= 2k \Sigma^2 + \tr(D)k\Sigma + (k^2-k)\Sigma^2 
%  = (k^2+k)\Sigma^2 + \tr(D)k\Sigma,
% \end{align*}
% where $\tr(\Sigma)=\tr(D)$ is valid due to the construction of $D$.
% \end{proof}
Now we are able to present an expression for the covariance matrix of the noise \eqref{covxi} in which the random variables $r^\l$ have been eliminated. Since $r^\l\sim N_n(0_n,\Sigma)$ we can define the random matrix $Q^\l:=r^\l(r^\l)^\T\sim W_n(\Sigma,1)$ and using \eqref{k1} we conclude
\begin{align*}
 \Cov(\noise^\l) &= A\E(r^\l(r^\l)^\T B r^\l(r^\l)^\T) A^\T + \Sigma - A\Sigma B\Sigma A^\T
 =A\E(Q^\l B Q^\l) A^\T + \Sigma - A\Sigma B\Sigma A^\T\\
 &= A U\left[ 2\left(\diag(D)(\diag(D))^\T\,\right)\circ (U^\T BU)+\tr(U^\T BU D)D \right] U^\T A^\T\\
 &\ \ \ \ + \Sigma - A\Sigma B\Sigma A^\T,
\end{align*}
where $D$ and $U$ satisfy $\Sigma=UDU^\T$ and $B=A^\T xx^\T A$.\\
In case of $\Sigma = I_n$ this expression can be simplified to
\begin{align*}
 \Cov(\noise^\l) &= A \left[ 2 B+\tr( B )I_n \right] A^\T + I_n - ABA^\T
 = 2ABA^\T +\tr(B)AA^\T + I_n - ABA^\T\\
 &= AA^\T xx^\T AA^\T + \tr(A^\T xx^\T A)AA^\T + I_n
 = AA^\T xx^\T AA^\T + \| A^\T x\|_2^2 AA^\T + I_n.
\end{align*}
As mentioned before, an alternative expression for $\E(QBQ)$ can be derived. If one chooses $A=I_k$, $B=I_k$, $\Psi=I_k$ and $\mu=0_{k\cdot n}$ in Theorem 2.2.9 (ii) from \cite{kollo2005advanced}, one obtains for $Q:=XX^\T$
\begin{align}
 \E(Q\otimes Q) = k^2\cdot \Sigma\otimes\Sigma + k\cdot\text{vec}(\Sigma)\text{vec}(\Sigma)^\T + k\cdot K_{n,n}\cdot \Sigma\otimes\Sigma,
\end{align}
where $K_{n,n}$ is the commutation matrix consisting of $n\times n$ blocks with $n\times n$ entries each. In the ($i,j$)-th block the only non-zero element is a ``1'' in position ($j,i$). It is a permutation matrix that can be used to describe the relationship between the vectorized forms of a square matrix $A$ and its transpose, since $\text{vec}(A^\T)=K_{n,n}\text{vec}(A)$. Using the calculation rule $(A\otimes B)\text{vec}(V)=\text{vec}(BVA^\T)$ for Kronecker products we get
\begin{align*}
 \E(QBQ)&=\text{mat}\left( \E(\text{vec}(QBQ)) \right)
 = \text{mat}(\E((Q\otimes Q)\text{vec}(B)))
 = \text{mat}(\E(Q\otimes Q)\text{vec}(B))\\
 &= \text{mat}((k^2\cdot \Sigma\otimes\Sigma + k\cdot\text{vec}(\Sigma)\text{vec}(\Sigma)^\T + k\cdot K_{n,n}\cdot \Sigma\otimes\Sigma)\text{vec}(B))\\
 &= \text{mat}(k^2\cdot\text{vec}( \Sigma B\Sigma) + k\cdot\text{vec}(\Sigma)\text{vec}(\Sigma)^\T\text{vec}(B) + k\cdot K_{n,n}\cdot \text{vec}(\Sigma B\Sigma))\\
 &= \text{mat}(\text{vec}(k^2 \Sigma B\Sigma) + k\cdot\text{vec}(\Sigma)\text{sum}(\Sigma\circ B) +  \text{vec}(k\Sigma B\Sigma))\\
 &= k^2 \Sigma B\Sigma + k\Sigma\cdot\text{sum}(\Sigma\circ B) + k\Sigma B\Sigma
 = k\cdot\tr(B\Sigma)\Sigma + (k^2+k)\Sigma B\Sigma,
\end{align*}
where $\text{sum}(M)$ denotes the sum over all entries of a matrix $M$. Thus, we get the same expression as in Theorem \ref{expqbq}.

\section{Numerical examples}\label{example}
\subsection{Illustrative example}
For dimension $n=10$ we generate randomly the matrices $B$ and $\Sigma$ with i.\,i.\,d. standard normally distributed entries and ensure, that $B$ is symmetric and $\Sigma\in S_+^n$. Let $k=3$ and $Q\sim W_n(\Sigma,k)$. The aim is to determine $\E(QBQ)$.
With the diagonal matrix $D$ and the orthogonal matrix $U$ from an eigendecomposition $\Sigma=UDU^\T$ and with Theorem \ref{expqbq} we get on the one hand
\begin{align*}%\label{etheo}
 E_{\text{exact}} := kU\left[ 2\left(\diag(D)(\diag(D))^\T\,\right)\circ (U^\T BU)+\tr(U^ \T BUD)D \right] U^\T+(k^2-k)\Sigma B\Sigma.
\end{align*}
On the other Hand, the expected value can be approximated with $m$ realizations $Q^i$ as
\begin{align}\label{eapprox}
 E_{\text{empiric}} := \frac1m\sum_{i=1}^m Q^i B Q^i,
\end{align}
because the law of large numbers provides
\begin{align*}
 \lim_{m\to\infty} \frac1m\sum_{i=1}^m Q^i B Q^i = \E(QBQ)\ \ \text{almost surely}.
\end{align*}
To get an impression of how fast $E_{\text{empiric}}$ approaches the theoretical value $E_{\text{exact}}$ for increasing sample size $m$, we plot the relative error $\|E_{\text{exact}}-E_{\text{empiric}}\|_2/\|E_{\text{exact}}\|_2$ against $m$.
Due to the randomness during the generation of $B$ and $\Sigma$ and in the realizations of $Q$, ten independent runs are made. At each run the relative error is calculated for $m\in\{1,10,100,10^3,10^4,10^5,10^6\}$. Thus, to be more precise, the logarithmic plot in figure \ref{fig:expvgl} shows the arithmetic means of the relative errors in dependence of $m$. 
\newpage
The standard deviation is represented by error bars:\\
\begin{figure}[h]
	\centering	\captionsetup{justification=centering,margin=2cm}
		\includegraphics[width=14cm]{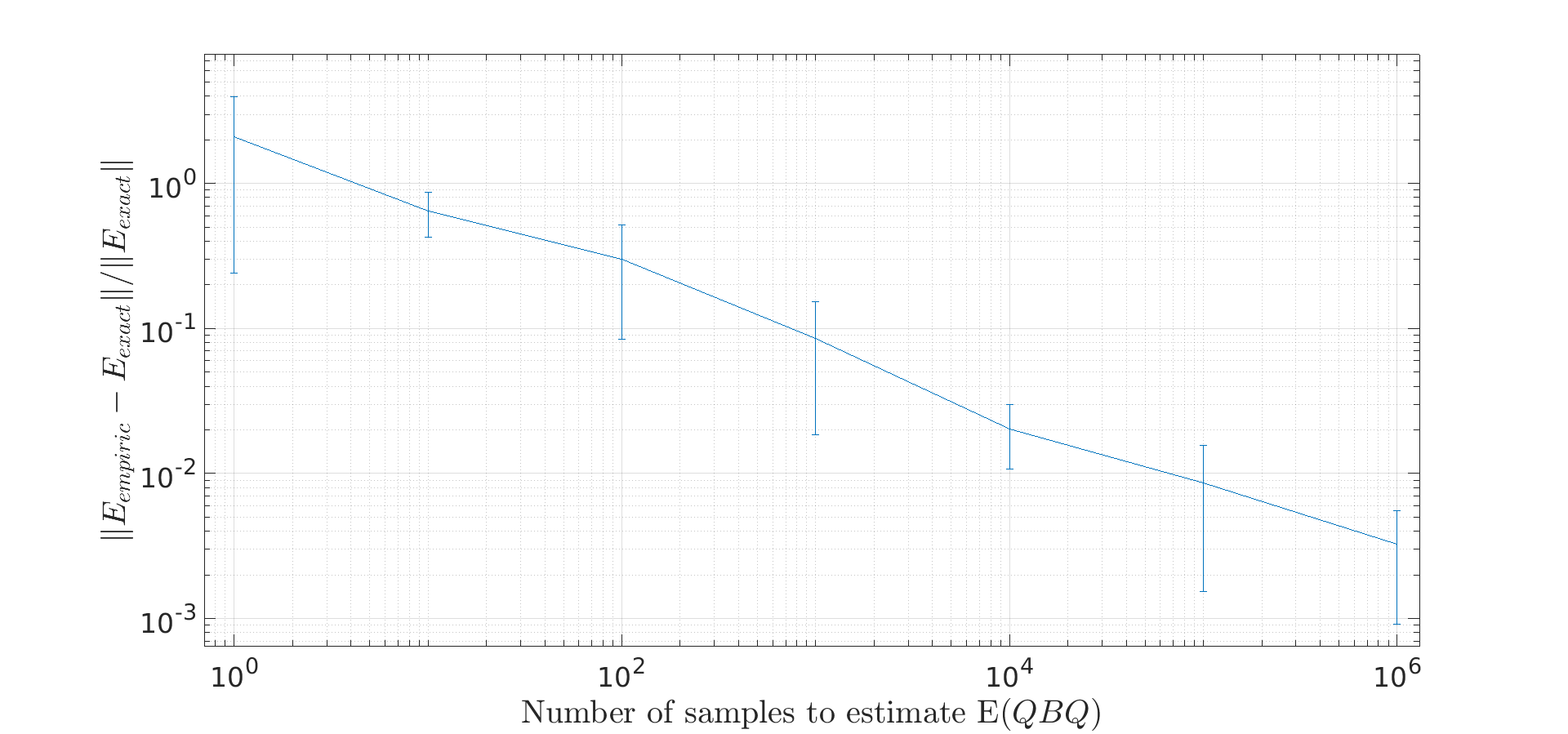}
    \caption{Relative error $\|E_{\text{exact}}-E_{\text{empiric}}\|_2/\|E_{\text{exact}}\|_2$ with respect\\ to the number of samples $m$ used for the approximation $E_{\text{empiric}}$. }
	\label{fig:expvgl}
\end{figure}\\
For $m=1$ the mean distance between $E_{\text{exact}}$ and $E_{\text{empiric}}$ is about $4\cdot 10^4$ which leads to a relative error of 2. Using $10^6$ samples this distance reduces to approximately 90 and the relative error to $3\cdot 10^{-3}$. The curve in the logarithmic plot is roughly linear decreasing. Two interesting observations arise: For large $m$ the approximation $E_{\text{empiric}}$ tends to the result of Theorem \ref{expqbq} and in order to approximate $\E(QBQ)$ adequately by \eqref{eapprox} many samples and a lot of time is needed. Thus, the main result is not only theoretically fascinating but also of practical relevance.

\subsection{Comparison of two SGD methods}
Our actual goal, as mentioned before, is to compare two algorithms that approximate solutions for the problem $\min_x f(x)$ with $f(x)=\lim_{m\to\infty}\sum_{\ell=1}^m f_\ell(x)$ from \eqref{objective}. For $n=10$ dimensions we randomly generate the entries of the matrices $\Sigma$ and $A$ i.\,i.\,d. from the $N(0,1)$ distribution and ensure that $\Sigma=\Sigma^\T$ is positive definite and that $A=A^T$ is positive semidefinite with norm $\|A\|_2=1$ and condition number $\text{cond}(A)=5$.\\
As initial value $x^0\in\R^n$ we choose the entries randomly from $N(0,1)$ and normalize the vector. Set the number of iterations to $\kmax=10^7$ and let the step length be given by $\gamma\equiv\gamma_k=10^{-3}$. Let $\{x^1,\dots,x^\kmax\}$ be the iterates generated by the SGD method \eqref{proc1}. A variation of the SGD method described above is the averaged SGD as analyzed in \cite{polyak}. Starting with $\xbar^{\,0}:=x^{\,0}$ the iterates of the ASGD can be defined as
\begin{align}
 \xbar^{\,k} := \frac{1}{k}\sum_{\ell=1}^k x^\l,
\end{align}
where $x^\l$ are the iterates of the ordinary SGD method and $k\in\{1,\dots,\kmax\}$. In each iteration we randomly draw $r^{\,k}$ and $b^{\,k}$ from $N_n(0_n,\Sigma)$, calculate $a^{\,k}=Ar^{\,k}$, the gradient of $f_k(x^{\,k})$ defined in \eqref{fl}, the iterates $x^{\,k}$ and $\xbar^{\,k}$ and the noise $\noise^{\,k}$ from \eqref{noise}. The necessary condition for a minimum of $f$ at $x^*$ is that the gradient has to vanish, i.\,e. $\nabla f(x^*)=0$. By construction the global optimal solution is $x^{\text{opt}}=0$. Below the two algorithms are compared by creating graphs of $\|\nabla f(x^{\,k})\|_2$ and $\|x^{\,k}-x^{\,\text{opt}}\|_2=\|x^{\,k}\|_2$ in dependence of the number of iterations $k$, respectively.\\
Additionally, we are interested in using our insights about the noise $\noise^{\,k}$ for this comparison. Since $E_{\text{exact}}
\equiv\E(\xi^{\,k})=0$ is valid independently of $k$, the approximation $E_{\text{empiric}}:=\tfrac{1}{k}\sum_{\ell=1}^k \noise^\l$ should tend to $E_{\text{exact}}$. This motivates plotting $\|E_{\text{exact}}-E_{\text{empiric}}\|_2$ with respect to the number of samples to estimate $E(\noise^k)$.\\
On the other hand, $\Cov(\noise^{\,k})$ depends on $k$. At the optimal solution $x^{\,\text{opt}}$ the covariance matrix of the noise is just the scale matrix $\Sigma$. We can use this to investigate, how the covariance matrices of the iterates, that can be calculated exactly using Theorem \ref{expqbq}, approach to $\Sigma$ with increasing number of iterations. Thus we plot $\|\Cov(\noise^{\,k})-\Sigma\|_2$ in dependence of the number of iterations $k$.
This way we obtain the following four graphs:
\begin{figure}[h]
	\centering	\captionsetup{justification=centering,margin=0.5cm}
    \includegraphics[width=16cm]{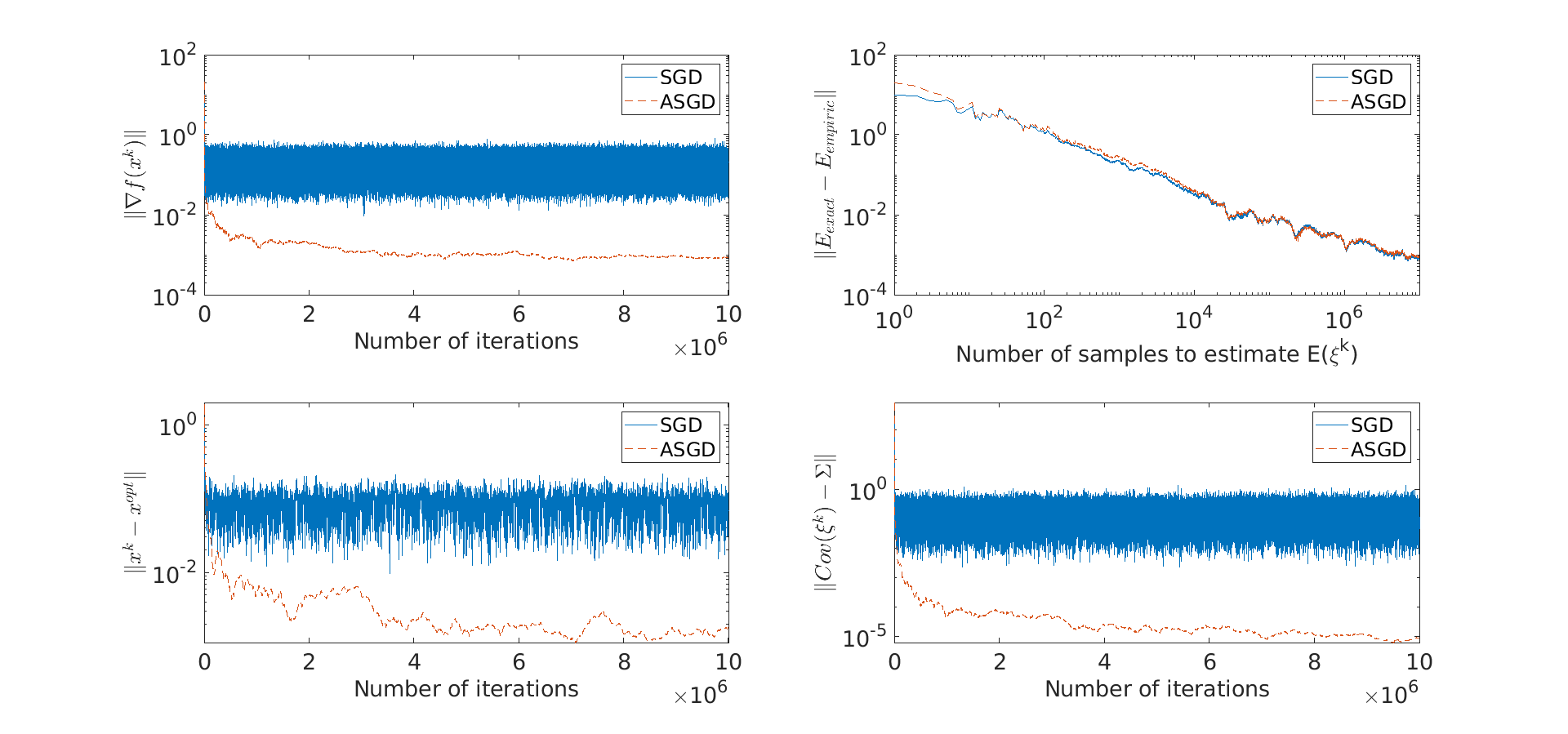}
    \caption{In all four graphs the iterates of the SGD method (blue, solid line)\\ is compared to the iterates of the ASGD method (red, dashed line).}
	\label{fig:sgdvgl}
\end{figure}\\
On the left hand side the norm of the gradient and the distance to the optimal solution are shown with respect to the number of iterations $k$. The ASGD method reaches lower values in both cases. In addition to that statistical fluctuations are much smaller.\\
At the top on the right hand side there is a plot of the distance of the empiric estimate of the expected value of the noise to the exact expected value in dependence of the number of samples. Both algorithms perform comparably well. Bottom right we have a plot of the distance of the covariance matrix of the noise in the $k^{\,\text{th}}$ iteration to the covariance matrix at the optimal solution with respect to $k$. Again, the ASGD method performs better in both counts: by reaching lower values and by fluctuating less.\\
The advantages of the ASGD are not surprising and consistent with the results of \cite{polyak}. This serves as a simple example of how two algorithms can be compared using Theorem \ref{expqbq}. 
%A different choice of the initial conditions and step lengths as well as further modifications have still potential to improve the performance of the algorithms. Furthermore, it would be interesting to study the algorithms in dependence of the dimension $n$ and the condition number of the matrix $A$. Of course, it should also be taken into account how good the algorithms are at more realistic examples.

\section{Conclusion}
In Theorem \ref{expqbq} it was proven that the expected value of the quadratic form $QBQ$ with $Q\sim W_n(\Sigma,k)$ and $B=B^\T$ can be expressed using $k$, $B$, $\Sigma$ and an eigendecomposition $\Sigma=UDU^\T$. Moreover, special cases for certain $k$, $\Sigma$ and $B$ were derived from this general formula, for instance the second momentum of a Wishart distributed random matrix $Q$, i.\,e. $\E(Q^2)$. A first example demonstrates the validity of the theorem. Beyond that the result is used to compare two stochastic methods. 

\subsection*{Acknowledgment}
I would like to express my sincere thanks to Florian Jarre, Holger Schwender and Dietrich von Rosen for their support and valuable comments.

\nocite{*}
\bibliography{literatur}
\bibliographystyle{natdin}

\end{document}